\documentclass[9pt,reqno,twoside,final]{amsart}
\usepackage{amsmath,amstext,amsthm,amsfonts,amssymb,amscd}

\numberwithin{equation}{section} 

\makeatletter 
\def\tagform@#1{\maketag@@@{$\langle$\ignorespaces#1\unskip\@@italiccorr$\rangle$}}

\newcommand{\N}{\mathbb{N}}

\renewcommand{\phi}{\varphi}
\renewcommand{\epsilon}{\varepsilon}
\renewcommand{\kappa}{\varkappa}       
\renewcommand{\ge}{\geqslant}           

\newtheorem{theorem}{Theorem}[section]
\newtheorem{proposition}{Proposition}[section]

\newtheorem{lemma}{Lemma}[section]

\theoremstyle{definition}

\theoremstyle{remark}

\theoremstyle{definition}

\newtheorem*{rem*}{Remark}
\newtheorem*{acknow*}{Acknowledgements}
\newtheorem*{examples*}{Examples}

\theoremstyle{plain}

\newtheorem*{theorem*}{Theorem}
\newenvironment{proof-sketch}{\noindent{\bf Sketch of Proof}\hspace*{1em}}{\qed\bigskip}
\newenvironment{proof-idea}{\noindent{\bf Proof Idea}\hspace*{1em}}{\qed\bigskip}
\newenvironment{proof-of-lemma}[1]{\noindent{\bf Proof of Lemma #1}\hspace*{1em}}{\qed\bigskip}
\newenvironment{proof-of-prop}[1]{\noindent{\bf Proof of Proposition #1}\hspace*{1em}}{\qed\bigskip}
\newenvironment{proof-of-thm}[1]{\noindent{\bf Proof of Theorem #1.}\hspace*{1em}}{\qed\bigskip}
\newenvironment{proof-attempt}{\noindent{\bf Proof Attempt}\hspace*{1em}}{\qed\bigskip}

\title[Finite Gr\"obner basis algebras with unsolvable problems]{Finite Gr\"obner basis algebras with unsolvable nilpotency problem 
and zero divisors problem}
\author{Ilya Ivanov-Pogodaev, Sergey Malev}
\address{Moscow Institute of Physics and Technology, Moscow, Russia }
\email{ivanov.pogodaev@gmail.com}
\address{School of mathematics, University of Edinburgh,
Edinburgh, UK}
\email{sergey.malev@ed.ac.uk}

\thanks{We would like to thank Agata Smoktunowicz and Alexei Kanel-Belov for interesting and fruitful discussions regarding this
paper}
\thanks{This research was supported by ERC Advanced grant Coimbra 320974.}

\thanks{This research was supported by Young Russian Mathematics award.}

\thanks{This research was supported by Russian Science Foundation  (grant no. 17-11-01377).}

\begin{document}
\maketitle

\renewcommand{\baselinestretch}{1.0}

\renewcommand{\labelenumi}{[\theenumi]}

\begin{abstract}
 This work presents a sample constructions of two algebras 
 both with the ideal of relations defined by a 
  finite Gr\"obner basis.  For the first algebra the question whether a given element is nilpotent is algorithmically unsolvable,
for the second one the question whether a given element is a zero divisor is algorithmically unsolvable. 
This gives a negative answer to questions raised by Latyshev.
  for which the question whether a given element is nilpotent is algorithmically
unsolvable. This gives a negative answer to a question raised by Latyshev.
\end{abstract}


\section{Introduction}

The word equality problem in finitely presented semigroups (and in algebras)
cannot be algorithmically solved.
This was proved in 1947 by Markov (\cite{Markov}) and independently by Post(\cite{Post}).
In 1952 Novikov constructed the first example of the group with unsolvable problem of word equality (see \cite{Novikov1} and \cite{Novikov2}).

In 1962 Shirshov proved solvability of the equality problem for Lie algebras with one relation
and raised a question about finitely defined Lie algebras (see \cite{Sh}).

In 1972 Bokut settled this problem.
In particular, he showed the existence of a finitely defined Lie algebra over an arbitrary field
with algorithmically unsolvable identity problem (\cite{Bokut}).

A detailed overview of algorithmically unsolvable problems can be found in  \cite{Ob}.

\medskip

Otherwise, some problems become decidable if
a finite Gr\"obner basis
defines a relations ideal.
In this case it is easy to determine whether two elements of the algebra are equal or not (see \cite{Bergman2}).

Gr\"obner bases for various structures are investigated by the Bokut
school in Guangzhou (\cite{BokutChen}).

In his work, Piontkovsky extended the concept of obstruction,
introduced by Latyshev (see \cite{Piont}, \cite{Piont2},
\cite{Piont3}, \cite{Piont4}).

Latyshev raised the question concerning the existence
of an algorithm that can find out if a given element is either a zero divisor
 or a nilpotent element when the ideal of
relations in the algebra is defined by a finite Gr\"obner basis.

Similar questions for monomial automaton algebras can be solved.
In this case the existence of an algorithm for nilpotent element or a zero divisor was
proved by Kanel-Belov, Borisenko and Latyshev \cite{BBL}. Note that these algebras are not Noetherian and not weak Noetherian.
Iyudu showed that the element property of being one-sided zero divisor is recognizable
in the class of algebras with a one-sided limited processing (see \cite{Iudu}, \cite{IuduDisser}).
It also follows from a solvability of a linear recurrence relations system on a tree (see \cite{Belov}).

\medskip

An example of an algebra with a finite Gr\"obner basis and algorithmically unsolvable problem
of zero divisor is constructed in \cite{IP}. 

A notion of {\it Gr\"obner basis} (better to say {\it Gr\"obner-Shirshov basis})
first appeared in the context of noncommutative (and not Noetherian) algebra. 
Note also that Poincar\'e-Birkhoff-Witt theorem can be canonically proved using Gr\"obner bases.
More detailed discussions of these questions see in \cite{Bokut}, \cite{Uf2}, \cite{BBL}.

In the present paper we construct an algebra with a finite Gr\"obner basis and algorithmically unsolvable problem of nilpotency.
We also provide a shorter construction for the zero divisors question.

For these constructions we simulate a universal Turing machine,
each step of which corresponds to a multiplication from the left by a chosen letter.

Thus, to determine whether an element is a zero divisor or is a nilpotent, it is not enough for an algebra to have
a finite Gr\"obner basis.

\section{The plan of construction}

Let $\Bbb A$ be an algebra over a field $\Bbb K$. Fix a finite alphabet of generators
$\{a_1,\dots , a_N\}$. A word in the alphabet of generators is called a {\it word in algebra}.

The set of all words in the alphabet is a semigroup. The main idea of the construction is a
realization of a universal Turing machine in the semigroup. We use the universal Turing
machine constructed by Marvin Minsky in \cite{Minsky}. This machine has $7$ states and $4$-color
tape. The machine can be completely defined by $28$ instructions.
Note that $27$ of them have a form
$$(i,j) \rightarrow (L,q(i,j),p(i,j)) \text{ or } (i,j) \rightarrow (R,q(i,j),p(i,j)),$$
where $0\leq i \leq 6$ is the current machine state, $0\leq j\leq 3$ is the current cell color,
$L$ or $R$ (left or right) is the direction of a head moving after execution of the current
instruction, $q(i,j)$ is the state after current instruction, $p(i,j)$ is the new color of
the current cell.

Thus, the instruction $(2,3) \rightarrow (L,3,1)$ means the following: ``If the color of the
current cell is $3$ and the state is $2$, then the cell changes the color to $1$, the head moves
one cell to the left, the machine changes the state to $3$.

The last instruction is $(4,3) \rightarrow\ $STOP. Hence, if the machine is in state $4$ and
the current cell has color $3$, then the machine halts.

\medskip

\subsection*{ Letters.}


By $Q_i$, $0\leq i\leq 6$ denote the current state of the machine. By $P_j$, $0\leq j\leq 3$
denote the color of the current cell.

The action of the machine depends on the current state $Q_i$ and current cell color $P_j$.
Thus every pair $Q_i$ and $P_j$ corresponds to one instruction of the machine.

The instructions moving the head to the left (right) are called {\it left} ({\it right}) ones.
Therefore there are {\it left pairs} $(i,j)$ for the left instructions, {\it right pairs} for
the right ones and instruction STOP for the pair $(4,3)$.


All cells with nonzero color are said to be {\it non-empty cells}.
We shall use letters $a_1$, $a_2$, $a_3$ for nonzero colors and letter $a_0$ for color zero.
Also, we use $R$ for edges of colored area.
Hence, the word $Ra_{u_1}a_{u_2}\dots a_{u_k}Q_iP_j a_{v_1}a_{v_2}\dots a_{v_l}R$ presents a full state of Turing machine.

We model head moving and cell painting using computations with powers of $a_i$ (cells) and $P_i$ and $Q_i$
(current cell and state of the machine's head).

\medskip


\section{Universal Turing machine}

We use the universal Turing machine constructed by Minsky. This machine is defined by the following instructions:

\medskip

$(0,0)\rightarrow (L,4,1) \ (0,1)\rightarrow (L,1,3) \ (0,2)\rightarrow (R,0,0) \ (0,3)\rightarrow (R,0,1)$

$(1,0)\rightarrow (L,1,2) \ (1,1)\rightarrow (L,1,3) \ (1,2)\rightarrow (R,0,0) \ (1,3)\rightarrow (L,1,3)$

$(2,0)\rightarrow (R,2,2) \ (2,1)\rightarrow (R,2,1) \ (2,2)\rightarrow (R,2,0) \ (2,3)\rightarrow (L,4,1)$

$(3,0)\rightarrow (R,3,2) \ (3,1)\rightarrow (R,3,1) \ (3,2)\rightarrow (R,3,0) \ (3,3)\rightarrow (L,4,0)$

$(4,0)\rightarrow (L,5,2) \ (4,1)\rightarrow (L,4,1) \ (4,2)\rightarrow (L,4,0) \ (4,3)\rightarrow $ STOP

$(5,0)\rightarrow (L,5,2) \ (5,1)\rightarrow (L,5,1) \ (5,2)\rightarrow (L,6,2) \ (5,3)\rightarrow (R,2,1)$

$(6,0)\rightarrow (R,0,3) \ (6,1)\rightarrow (R,6,3) \ (6,2)\rightarrow (R,6,2) \ (6,3)\rightarrow (R,3,1)$

\medskip

We use the following alphabet:
$$\{t, \, a_0, \dots  a_3, \, Q_0, \dots Q_6, \, P_0 \dots P_3, \,  R\}$$

For every pair except $(4,3)$ the following functions are defined:
$q(i,j)$ is a new state, $p(i,j)$ is a new color of the current cell (the head leaves it).

\medskip

\section{Defining relations for the nilpotency question}

Consider the following defining relations:

\begin{eqnarray}
   &
tRa_l=Rta_l;  \text{\ \  $0\leq l\leq 3$}                           \label{tt1}  \\ &
ta_lR=a_lRt;  \text{\ \  $0\leq l\leq 3$}                           \label{tt1b}  \\ &
ta_ka_j=a_kta_j;                           \quad \text{ $0\leq k,j\leq 3$}    \label{tt2} \\ &
ta_kQ_iP_j=Q_{q(i,j)}P_kta_{p(i,j)}; \text{for left pairs $(i,j)$ and $0\leq k\leq 3$} \label{tt3} \\ &
tRQ_iP_j=RQ_{q(i,j)}P_0ta_{p(i,j)}; \text{for left pairs $(i,j)$ and $0\leq k\leq 3$} \label{tt5} \\ &
ta_lQ_iP_ja_ka_n=a_la_{p(i,j)}Q_{q(i,j)}P_kta_n;\text{for right pairs $(i,j)$ and $0\leq k\leq 3$} \label{tt4} \\ &
ta_lQ_iP_ja_kR=a_la_{p(i,j)}Q_{q(i,j)}P_kRt;\text{for right pairs $(i,j)$ and $0\leq k\leq 3$} \label{tt4r} \\ &
tRQ_iP_ja_ka_n=Ra_{p(i,j)}Q_{q(i,j)}P_kta_n;\text{for right pairs $(i,j)$ and $0\leq k\leq 3$} \label{tt4b} \\ &
tRQ_iP_ja_kR=Ra_{p(i,j)}Q_{q(i,j)}P_kRt;\text{for right pairs $(i,j)$ and $0\leq k\leq 3$} \label{tt4ar} \\ &
ta_lQ_iP_jR=a_la_{p(i,j)}Q_{q(i,j)}P_0Rt; \text{for right pairs $(i,j)$ and $0\leq l\leq 3$} \label{tt6}  \\ &
tRQ_iP_jR=Ra_{p(i,j)}Q_{q(i,j)}P_0Rt; \text{for right pairs $(i,j)$} \label{tt6b}  \\ &
Q_4P_3=0.                                  \label{tt7}
\end{eqnarray}

\medskip

The relations \eqref{tt1} and \eqref{tt2} are used to move $t$ from the left edge to the last letter $a_l$ standing before 
$Q_iP_j$ which represent the head of the machine.  The relations \eqref{tt3}--\eqref{tt6b}
represent the computation process. The relation \eqref{tt1b} is used to move $t$ through the finishing letter $R$.

Finally, the relation \eqref{tt7} halts the machine.

\section{Nilpotency of the fixed word and machine halt}

Let us call the word $tRa_{u_1}a_{u_2}\dots a_{u_k}Q_iP_j a_{v_1}a_{v_2}\dots a_{v_l}R$ {\it the main word}.
The main goal is to prove the following theorem:

\begin{theorem} \label{th1}
The machine halts if and only if the main word is nilpotent in
the algebra presented by the defining relations \eqref{tt1}--\eqref{tt7}.
\end{theorem}

First, we prove some propositions.
\begin{rem*}
 We use sign $\equiv$ for lexicographical equality and sign $=$ for equality in algebra.
\end{rem*}


\medskip

Consider a full state of our Turing machine represented by the word
$$Ra_{u_1}a_{u_2}\dots a_{u_k}Q_iP_j a_{v_1}a_{v_2}\dots a_{v_l}R.$$

Suppose that $U\equiv a_{u_1}a_{u_2}\dots a_{u_k}$  and $V\equiv a_{v_1}a_{v_2}\dots a_{v_l}$.
Therefore $U$ and $V$ represent the colors of all cells on the Turing machine tape.
We denote the full state of this machine as $M(i,j,U,V)$.
Suppose that $M(i',j',U',V')$ is the next state ($M(i,j,U,V)\rightarrow M(i',j',U',V')$).

\medskip

Consider a semigroup $G$ presented by the defining
relations \eqref{tt1}--\eqref{tt7}.
Suppose that $W(i,j,U,V)$ is a word in $G$ corresponding to machine state $M(i,j,U,V)$.
(Actually $W(i,j,U,V)\equiv Ra_{u_1}a_{u_2}\dots a_{u_k}Q_iP_j a_{v_1}a_{v_2}\dots a_{v_l}R$.)

\begin{proposition} \label{ll2}
Let us move all the words from the relations \eqref{tt1}--\eqref{tt7} to the left-hand side.
There exists a reduction order on the free monoid generated by alphabet  $\Phi=\{t, \, a_0, \dots,  a_3, \, Q_0, \dots, Q_6, \, P_0, \dots, P_3, \,  R\}$, 
such that 
the left-hand sides of the obtained equalities comprise a Gr\"obner basis in
the ideal generated by them. 
%
\end{proposition}

\begin{proof}
Recall, that by reduction order on the free monoid $\Phi^*$ we mean a well order such that the empty word is the minimal one, 
and for any $a,b,s_1,s_2\in\Phi^*$, if $s_1\prec s_2$ then $as_1b\prec as_2b$.

Any word $w$ from $\Phi^*$ can be uniquely written as $X_0tX_1t\cdots tX_n$, where $X_i\in\Phi^*$ are free from the letter  $t$. 
Each $X_i$ can be empty, even all of them (if the word is $t^n$).
By {\it height} of this word we call 
$$h(w)=\sum\limits_{i=0}^n 2^i\deg X_i.$$

We define the following order. Given two words $w_1$ and $w_2$, we compare them with respect to the degree of $t$. 
If $\deg_t(w_1)<\deg_t(w_2)$ then $w_1\prec w_2$.
If $\deg_t(w_1)=\deg_t(w_2)$ then we compare them with respect to the height. 
If $h(w_1)<h(w_2)$ then $w_1\prec w_2$.
If their heights are also equal then we use a deglex order to compare them.

We need to prove that this order is a reduction order. 

Note that an empty word is the minimal (it has a zero degree of $t$, a zero height and a zero degree).

Assume $a,b,s_1,s_2\in\Phi^*$ and $s_1\prec s_2$.

If $\deg_t(s_1)<\deg_t(s_2)$ , then  $\deg_t(as_1b)<\deg_t(as_2b)$, therefore $as_1b\prec as_2b$.

Assume $\deg_t(s_1)=\deg_t(s_2)=n$ and $h(s_1)<h(s_2)$. 
In this case we will show that multiplication of inequality by one symbol does not change it.
In other words, we will show for any symbol $x\in\Phi$ that $xs_1\prec xs_2$ and $s_1x\prec s_2x$.
First assume that $x\neq t$. 
Then multiplication by $x$ from the left increases a height by $1$ of both sides, thus an inequality remains.
Note that multiplication by $x$ from the right increases a height by $2^n$ of both sides, and an inequality remains also.
The multiplication by $t$ from the left multiplies both heghts by $2$ and multiplication by $t$ from the right does not change it.

Now assume that  $\deg_t(s_1)=\deg_t(s_2)$ and $h(s_1)=h(s_2)$. 
Hence $\deg_t(as_1b)=\deg_t(as_2b)$ and $h(as_1b)=h(as_2b)$. 
In this case we compare both pairs $(s_1,s_2)$ and $(as_1b,as_2b)$ by deglex order which is a reduction order.

Note that every left-hand side contains a leading monomial.
There is no such word that begins some leading monomial in the basis and ends some other leading monomial.
\end{proof}

\begin{lemma}\label{wordend}
 For any nonempty word $U\equiv a_{i_1}\cdots a_{i_l}$ we have $tUR=URt$.
\end{lemma}
\begin{proof}
We can use the relation \eqref{tt2} $(l-1)$ times and transform $tUR$ to $a_{i_1}\cdots a_{i_{l-1}}ta_{i_l}R$.
After that we use the relation \eqref{tt1b}.
\end{proof}

\begin{proposition} \label{ll1}
If $i=4$ and $j=3$ then $tW(i,j,U,V) = 0$. Otherwise, the following condition holds:  $tW(i,j,U,V) = W(i',j',U',V')t$.

\end{proposition}

\begin{proof}

Consider the word $tW(i,j,U,V)=tRUQ_iP_jVR$. If $i=4$ and $j=3$ then we can apply relation \eqref{tt7}.
Otherwise, suppose that $(i,j)$ is a left pair.

If $U$ is not empty word, then we can write $U=\tilde Ua_k$ for some $0\leq k\leq 3$.
In this case we have the word $tR\tilde Ua_kQ_iP_jVR$. We can use the relation \eqref{tt1} to transform it to
$Rt\tilde Ua_kQ_iP_jVR$. Now we use the relation \eqref{tt2} the degree of $\tilde U$ times: our words transforms to
$R\tilde Uta_kQ_iP_jVR$. After that we use the relation \eqref{tt3} and our word transforms to
$R\tilde U Q_{i'}P_{j'}ta_{p(i,j)}VR$. Now we use Lemma \ref{wordend}.

If $U$ is empty, then $tW(i,j,U,V)\equiv tRQ_iP_jVR$.
In this case we will start our chain with using relation \eqref{tt5}:
$tRQ_iP_jVR=RQ_{i'}P_0ta_{p(i,j)}VR$. 
After that we use Lemma \ref{wordend}.
\medskip

Suppose that $(i,j)$ is a right pair. In this situation we will have six cases:

{\bf Case 1}
$U$ and $V$ are empty words.
In this case our word is $tRQ_iP_jR$ and we use the relation \eqref{tt6b}.

{\bf Case 2} $U$ is empty and $V=a_k$ is a word of degree $1$.
In this case our word is $tRQ_iP_ja_kR$, and we can use a relation \eqref{tt4ar}.

{\bf Case 3} $U$ is empty and $V=a_k\tilde V$ is a word of degree greater than $1$.
In this case our word is $tRQ_iP_ja_k\tilde V R$. We can use a relation \eqref{tt4b} to transform it to
$Ra_{p(i,j)}Q_{i'}P_{j'}t\tilde V R$, where $\tilde V$ is not empty. Thus we can use Lemma \ref{wordend} to complete the chain.

{\bf Case 4} $U=\tilde Ua_l$ is not empty and $V$ is empty.
In this case our word is $tR\tilde Ua_lQ_iP_jR$. We use a relation \eqref{tt1} and transform it to
$Rt\tilde Ua_lQ_iP_jR$. Using relation \eqref{tt2} the degree of $\tilde U$ times will transform our word to
$R\tilde U ta_lQ_iP_jR$. A relation \eqref{tt6} completes a chain.

{\bf Case 5} $U=\tilde Ua_l$ is not empty and $V=a_k$ is a word of degree $1$.
In this case our word is $tR\tilde Ua_lQ_iP_ja_kR$. Similar to Case 4 we can transform our word to
$R\tilde U ta_lQ_iP_ja_kR$. A relation \eqref{tt4r} completes a chain.

{\bf Case 6} $U=\tilde Ua_l$ is not empty and $V=a_k\tilde V$ is a word of degree greater than $1$.
In this case our word is $tR\tilde Ua_lQ_iP_ja_k\tilde VR$. Similar to Case 4 we can transform our word to
$R\tilde U ta_lQ_iP_ja_k\tilde VR$. A relation \eqref{tt4} transforms it to
$R\tilde U a_la_{p(i,j)}Q_{y'}P_{j'}t\tilde VR$. Now we use lemma \ref{wordend} to complete our chain.
\end{proof}

\medskip

\begin{proposition} \label{ll3}
The following statements are equivalent:

\begin{enumerate}
\item[(i)] The Turing machine described above begins with the state $M(i,j,U,V)$ and
halts in several steps.

\item[(ii)] There exists a positive integer $N$ such that
$t^N RUQ_iP_jVR= 0$.
\end{enumerate}

\end{proposition}

\begin{proof}
First, prove that second statement is a consequence of the first one.

Suppose that $M(i,j,U,V)$ transforms to $M(4,3,U',V')$ in one step. According to
Proposition~\ref{ll1} $tW(i,j,U,V) = W(4,3,U',V')t$.
Then we can apply $Q_4P_3=0$ by \eqref{tt7} and obtain zero.

\medskip
Suppose that the statement is true for $m$ (and fewer) steps.
Let the machine begin with
state $M(i,j,k,n)$ and halt after $m+1$ step.
Consider the first step in the chain.
Let it be the step from  $M(i,j,U,V)$ to $M(i',j',U',V')$.
Apply Proposition \ref{ll1} for this step. Hence
$tRUQ_iP_jVR = RU'Q_{i'}P_{j'}V'Rt.$

The machine started in the state $M(i',j',U',V')$ halts in $m$ steps. Using induction we complete the proof.

\medskip

Now let us prove that the first statement is a consequence of the second one.




If $t^N RUQ_iP_jVR= 0$, then there exists a chain of equivalent words, starting with $t^N RUQ_iP_jVR$ and finishing with $0$.
The only way to obtain $0$ is to use a relation $Q_4P_3=0$. Therefore the word before $0$ in the chain contains $Q_4P_3$.

By {\it structure of the word $W$}, $S(W)$ let us denote the word $W$, where all letters $t$ will be deleted.
Each word in the chain will have a structure $RU_kQ_{i_k}P_{j_k}V_kR$ because the only relation that breaks this structure is 
$Q_4P_3=0$, and it will be used only one time, in the end of the chain. Note that each structure corresponds to the Turing machine.
The only way to obtain $0$ in this chain is to change indices of $Q$ and $P$ in the structure. This can be done by moving $t$. 

According to the Proposition \ref{ll1}, moving $t$ from the left to the right corresponds to the Turing Machine's one step 
to the future, 
and moving $t$ from the right to the left corresponds to the Turing Machine's one step to the past 
(note that this is not always possible). 
There is a Gr\"obner basis of relations in our algebra, thus we can assume that in our chain  words decrease
(each word is lower than the previous with respect to the order on the free monoid $\Phi^*$). 
Therefore letters $t$ move only from the left to the right.


Hence there exists $k\leq N$ such that $t^N RUQ_iP_jVR=t^{N-k}R\tilde U Q_4P_3\tilde VRt^k$.

Therefore the machine halts after $k$ steps.

\end{proof}

Now we are ready to prove the theorem above.

\begin{theorem} \label{th1}
Consider an algebra $A$ presented by the defining
relations \eqref{tt1}--\eqref{tt7}.
The word $tRUQ_iP_jVR$ is nilpotent in $A$ if and only if machine $M(i,j,U,V)$ halts.
\end{theorem}

\begin{proof}
Suppose that $(tRUQ_iP_jVR)^n = 0$.
The structure of this word corresponds to a row of $n$ separate machines. Using relations we can transform some machine to the next
state (note that we have a Gr\"obner basis in the algebra, therefore we can assume that words in the chain will decrease).
Thus if we obtain $Q_4P_3$ for some machine, we can conclude that this machine halts after several steps. 
Therefore $M(i,j,U,V)$ halts.

\medskip

Suppose that $M(i,j,U,V)$ halts. Then $t^n RUQ_iP_jVR= 0$ for some minimal $n$.
We can obtain $(tRUQ_iP_jVR)^n= At^nRUQ_iP_jVR$ (for some word $A$) by using Proposition \ref{ll1} several times.
Therefore $(tRUQ_iP_jVR)^n = 0$.
\end{proof}

\medskip

Since the halting problem cannot be algorithmically solved, the nilpotency problem in algebra $A$ is algorithmically unsolvable.

\medskip

\section{Defining relations for a zero divisors question}

We use the following alphabet:
$$\Psi=\{t, \, s, \, a_0, \dots  a_3, \, Q_0, \dots Q_6, \, P_0 \dots P_3, \,  L, \,  R \}.$$

For every pair except $(4,3)$ the following functions are defined:
$q(i,j)$ is a new state, $p(i,j)$ is a new color of the current cell (the head leaves it).

\medskip

Consider the following defining relations:

\begin{eqnarray}
   &
tLa_k=Lta_k; \text{ $0\leq k\leq 3$}                             \label{td1}  \\ &
ta_ka_l=a_kta_l;                           \quad \text{ $0\leq k,l\leq 3$}    \label{td2} \\ &
sR=Rs;         \label{td9}  \\ &
sa_k=a_ks;    \text{ $0\leq k\leq 3$}  \label{td8} \\ &
ta_kQ_iP_j=Q_{q(i,j)}P_ka_{p(i,j)}s; \text{for left pairs $(i,j)$ and $0\leq k\leq 3$} \label{td3} \\ &
tLQ_iP_j=LQ_{q(i,j)}P_0a_{p(i,j)}s; \text{for left pairs $(i,j)$} \label{td5} \\ &
ta_lQ_iP_ja_k=a_la_{p(i,j)}Q_{q(i,j)}P_ks;\text{for right pairs $(i,j)$ and $0\leq k,l\leq 3$} \label{td4} \\ &
tLQ_iP_ja_k=La_{p(i,j)}Q_{q(i,j)}P_ks;\text{for right pairs $(i,j)$ and $0\leq k\leq 3$} \label{td4b} \\ &
ta_lQ_iP_jR=a_la_{p(i,j)}Q_{q(i,j)}P_0Rs; \text{for right pairs $(i,j)$ and  $0\leq l\leq 3$ } \label{td6}  \\ &
tLQ_iP_jR=La_{p(i,j)}Q_{q(i,j)}P_0Rs; \text{for right pairs $(i,j)$} \label{td6b}  \\ &
Q_4P_3=0;                                  \label{td7}
\end{eqnarray}

\medskip

The relations \eqref{td1}--\eqref{td2} are used to move $t$ from the left edge to the letters
$Q_i$, $P_j$ which present the head of the machine.
The relations \eqref{td9}--\eqref{td8} are used to move $s$ from the letter $Q_i$, $P_j$ to the right edge.
The relations \eqref{td3}--\eqref{td6} represent the computation process. Here we use relations of the form $tU=Vs$.

Finally, the relation \eqref{td7} halts the machine.

\section{Zero divisors and machine halt}

Let us call the word $La_{u_1}a_{u_2}\dots a_{u_k}Q_iP_j a_{v_1}a_{v_2}\dots a_{v_l}R$ {\it the main word}.
The main goal is to prove the following theorem:

\begin{theorem} \label{th2}
The machine halts if and only if the main word is a zero divisor in
the algebra presented by the defining relations \eqref{td1}--\eqref{td7}.
\end{theorem}


\medskip

Consider a full state of our Turing machine represented by the word
$$La_{u_1}a_{u_2}\dots a_{u_k}Q_iP_j a_{v_1}a_{v_2}\dots a_{v_l}R.$$

Suppose that $U=a_{u_1}a_{u_2}\dots a_{u_k}$  and $V=a_{v_1}a_{v_2}\dots a_{v_l}$.
Therefore $U$ and $V$ represent the colors of all cells on the Turing machine tape.
We denote the full state of this machine as $T(i,j,U,V)$.
Suppose that $T(i',j',U',V')$ is the next state ($T(i,j,U,V)\rightarrow T(i',j',U',V')$).

\medskip

Consider a semigroup $S$ presented by the defining
relations \eqref{td1}--\eqref{td7}.
Suppose that $F(i,j,U,V)$ is a word in $S$ corresponding to machine state $T(i,j,U,V)$.

\begin{proposition} \label{ld2}
Let us move all the words from relations \eqref{td1}--\eqref{td7} to the left-hand side.
Consider the semi-DEGLEX order: $\{t, \, s, \, a_0, \dots  a_3, \, Q_0, \dots, Q_6, \, P_0, \dots, P_3, \,  L, \,  R\}.$ 
The left-hand sides of the obtained equalities comprise a Gr\"obner basis in
the ideal generated by them.
%
\end{proposition}

\begin{proof}
We will use a weighted degree instead of the usual: each letter from the alphabet (except for $t$) 
will have degree $1$, however the degree of $t$ equals $2$. (For example, $\deg(tRL)=4$)

This order is a reduction order.

Note that every left-hand side contains a leading monomial.
There is no such word that begins some leading monomial in the basis and ends some other leading monomial.
\end{proof}

\medskip

\begin{proposition} \label{ld1}
If $i=4$ and $j=3$ then $tF(i,j,U,V) = 0$. Otherwise, the following condition holds:  $tF(i,j,U,V) = F(i',j',U',V')s$.

\end{proposition}

\begin{proof}

Consider the word $tF(i,j,U,V)=tLUQ_iP_jVR$. If $i=4$ and $j=3$ then we can apply relation \eqref{td7}.
Otherwise, suppose that $(i,j)$ is a left pair.

If $U$ is an empty word then $tF(i,j,U,V)=tLQ_iP_jVR$.
Hence we can apply relation \eqref{td5} to obtain $tLQ_iP_jVR= LQ_{q(i,j)}P_0a_jsVR$.
Using \eqref{td9} and \eqref{td8} we finally have
$$tLQ_iP_jVR= LQ_{q(i,j)}P_0a_jsVR= LQ_{q(i,j)}P_0a_jVRs.$$
According to the definition of $q(i,j)$ and $p(i,j)$, the word $LQ_{q(i,j)}P_0a_jVR$ corresponds to the next state of the machine.

If $U$ is not an empty word, we can write $U=U_1a_k$ for some $k$. 
We use the relations \eqref{td1} and \eqref{td2} and obtain that $tLUQ_iP_jVR= LU_1ta_kQ_iP_jVR$. 
Further, we use relation \eqref{td3}:  
$LU_1ta_kQ_iP_jVR = LU_1Q_{q(i,j)}P_ka_{p(i,j)}VRs$. 
The word $LU_1Q_{q(i,j)}P_ka_{p(i,j)}VR$ corresponds to the next state of the machine.

Assume that $(i,j)$ is a right pair.
If $U$ and $V$ are empty words, than we use relation \eqref{td6b}.

If $U$ is empty, and $V=a_k\tilde V$ is not, then we use the relation \eqref{td4b} and obtain
$tLQ_iP_ja_k\tilde VR =La_{p(i,j)}Q_{q(i,j)}P_ks \tilde VR$. After that we use relations \eqref{td8} and \eqref{td9} and move 
$s$ to the right.

Assume $U=\tilde U a_k$ is not empty.
In this case we use the relation \eqref{td2} the length of $\tilde U$ times and obtain $t\tilde U a_k=\tilde U t a_k$.
If $V$ is empty we can use the relation \eqref{td6}. 
If $V=a_l\tilde V$ is not empty then we can use the relation \eqref{td4}, after that we will use relations \eqref{td8} and \eqref{td9} and move 
$s$ to the right.
\end{proof}

\medskip

\begin{proposition} \label{ld3}
The following statements are equivalent:

\begin{enumerate}
\item[(i)] The Turing machine described above begins with the state $T(i,j,U,V)$ and
halts in several steps.

\item[(ii)] There exists a positive integer $N$ such that
$t^N LUQ_iP_jVR= 0$.
\end{enumerate}

\end{proposition}

\begin{proof}
First, prove that second statement is a consequence of the first one.

Suppose that $T(i,j,U,V)$ transforms to $T(4,3,U',V')$ in one step.
According to Proposition~\ref{ld1} $tF(i,j,U,V) = F(4,3,U',V')s$.
Then we can apply $Q_4P_3=0$ by \eqref{td7} and obtain zero.

\medskip
Suppose that the statement is true for $m$ (and fewer) steps.
Let the machine begin with
state $T(i,j,k,n)$ and halt after $m+1$ step.
Consider the first step in the chain.
Let it be the step from  $T(i,j,U,V)$ to $T(i',j',U',V')$.
Apply Proposition \ref{ld1} for this step. Hence
$tLUQ_iP_jVR = LU'Q_{i'}P_{j'}V'Rs.$

The machine started in the state $T(i',j',U',V')$ halts in $m$ steps. Using induction we complete the proof.

\medskip

Now let us prove that the first statement is a consequence of the second one.




If $t^N LUQ_iP_jVR= 0$, then there exists a chain of equivalent words, starting with $t^N LUQ_iP_jVR$ and finishing with $0$.
The only way to obtain $0$ is to use a relation $Q_4P_3=0$. Therefore the word before $0$ in the chain contains $Q_4P_3$.

By {\it structure of the word $W$}, $S(W)$ let us denote the word $W$, where all letters $t$ and $s$ will be deleted.
Each word in the chain will have a structure $LU_kQ_{i_k}P_{j_k}V_kR$ because the only relation that breaks this structure is 
$Q_4P_3=0$, and it will be used only one time, in the end of the chain. Note that each structure corresponds to the Turing machine.
The only way to obtain $0$ in this chain is to change indices of $Q$ and $P$ in the structure. This can be done by moving $t$. 

According to the Proposition \ref{ld1}, moving $t$ from the left to the right, and transforming it to $s$ corresponds to the 
Turing Machine's one step. 
Note that there is a Gr\"obner basis on our algebra, thus we can assume that words in the chain decrease. 
In particular, moving $s$ from the right to the left, transforming it to $t$ is impossible.


Therefore we can obtain $Q_4P_3$ only by moving $t$ from the left to $s$ on the right,
and there exists $k\leq N$ such that $t^N LUQ_iP_jVR=t^{N-k}L\tilde U Q_4P_3\tilde VRt^k$.

Therefore the machine halts after $k$ steps.
\end{proof}

\begin{proposition} \label{ld4}
If $Xt= 0$ in $S$, then $X= 0$. If $sX= 0$ in $S$, then $X= 0$.

\end{proposition}

\begin{proof}
Suppose that we apply some relations and transform $Xt$ to zero.

We say that the letter $t$ is {\it almost last} if the word has the form $Y_1tY_2$,
and $Y_2$ contains $a_k$ and $L$ letters only.
Note that if an almost last $t$-letter occurs in some relation then this relation is \eqref{td1} or \eqref{td2}.
Therefore that $t$-letter is always almost last.
It is clear that an almost last $t$-letter always exists in every word which is equivalent to $Xt$.
Since an almost last $t$-letter never participates in relations  \eqref{td9}--\eqref{td7},
we can situate it on the right edge of we word $Xt$ while we use our relations.
We did not use the $t$-letter, and therefore we can do the same with the word $X$.

Similarly we can prove that if $sX= 0$ then $X= 0$.
\end{proof}

\begin{proposition} \label{ld5}
If $Xt^n= 0$ in $S$, then $X= 0$. If $s^nX= 0$ in $S$, then $X= 0$.
\end{proposition}

\begin{proof}
We can prove this by induction.
\end{proof}

Now we are ready to prove the theorem above.

\begin{theorem} \label{th2}
Consider an algebra $H$ presented by the defining
relations \eqref{td1}--\eqref{td7}.

The word $LUQ_iP_jVR$ is a zero divisor in the algebra $H$ if and only if machine $T(i,j,U,V)$ halts.
\end{theorem}

\begin{proof}
Suppose that machine $T(i,j,U,V)$ halts. Using Proposition \ref{ld3} we have $t^N LUQ_iP_jVR= 0$ for some positive integer $N$.
Thus, the word $LUQ_iP_jVR$ is a zero divisor.

\medskip

Let $XLUQ_iP_jVRY = 0$ for some algebra elements $X,Y \ne 0$. Suppose that $X$, $Y$ are some words.

Note that $L$ and $R$ letters cannot disappear from the word.
Hence we can divide our word into three parts: to the left of $L$, to the right of $R$, and between $L$ and $R$.
There is only one relation which can turn the word $XLUQ_iP_jVRY$ to zero: $Q_4P_3=0$.
Thus this subword $Q_4P_3$ can appear in three possible parts of the word.
Note that only $t$ letters can pass through $L$ and only $s$ letters can pass through $R$.
Every relation can change nothing in the area to the left side of $L$ and to the right side of $R$, except $t$ and $s$-letters occurrences.
Therefore if $Q_4P_3$ appears to the left of $L$, then $Xs^n= 0$. Using Proposition~\ref{ld5} we obtain a contradiction: $X= 0$. Similarly if  $Q_4P_3$ appears to the right of $R$, then $Y= 0$. Thus $Q_4P_3$ appears between $L$ and $R$.

\medskip
Consider the {\it structure} of the word $LUQ_iP_jVR$.
For any structure of a word equivalent to $LUQ_iP_jVR$ there exists a corresponding state of the machine.
Since only $t$ letters can pass through $L$ and only $s$ letters can pass through $R$,
we can change the structure of the word $LUQ_iP_jVR$ by turn to the next or the previous machine state.
If $Q_4P_3$ appears between $L$ and $R$ then we can obtain a STOP state. Thus the machine $T(i,j,U,V)$ halts.

\medskip

Now let us consider the general case: $X$, $Y$ are some algebra elements.
Suppose that $X=c_1X_1+\dots c_nX_n$, $Y=d_1Y_1+\dots d_mY_m$, where
$X_k$, $Y_l$ are words, and $c_k$ and $d_l$ are elements of the field. Without loss of generality we may assume that
$n$ is  the minimal possible, and for this $n$ $m$ is the minimal possible. We also may assume that $X_k$, $Y_l$ are written 
in the reduced form. We assume that either $n>1$, or $m>1$.

Consider the function $\tilde h:\Psi^*\rightarrow\N_0$: for any word $w$ $\tilde h(w)=\deg_t(w)+\deg_s(w)$.
Note that relations \eqref{td1}-\eqref{td6b} do not change value of $\tilde h$, therefore it is invariant under the word reduction.
Assume that $\tilde h(X_{k_1})\neq\tilde h(X_{k_2})$. 
In this case we will take a subset $S_x\subseteq\{1,\dots,n\}$ such that $\tilde h$ takes a maximal value on $X_k$ for $k\in S_x$.
We also take a subset $S_y\subseteq\{1,\dots,m\}$ such that $\tilde h$ takes a maximal value on $Y_l$ for $l\in S_y$.
We know that $(c_1X_1+\dots c_nX_n)W(d_1Y_1+\dots d_mY_m)=\sum_{k,l} c_kd_l X_kLUQ_iP_jVRY_l=0$.
Therefore one can reduce this element to zero. 
However none of the elements $X_kLUQ_iP_jVRY_l$ can be reduced to zero.
Thus, all elements $X_kLUQ_iP_jVRY_l$ can be separated to several sets of similar words.
Note that all words $X_kLUQ_iP_jVRY_l$ (where $k\in S_x$ and $l\in S_y$) can be similar only to a word $X_{k'}LUQ_iP_jVRY_{l'}$
where $k'\in S_x$ and $l'\in S_y$. Hence, $(\sum\limits_{k\in S_x}c_kX_k)\cdot LUQ_iP_jVR \cdot (\sum\limits_{l\in S_y}d_lY_l)=0$.
A contradiction ($n$ was taken as a minimal possible).
Therefore $\tilde h(X_k)$ does not depend on $k$ and $\tilde h(Y_l)$ does not depend on $l$.

We have $XLUQ_iP_jVRY = \sum_{k,l} c_kd_l X_kLUQ_iP_jVRY_l$.
We can consider our defining relations as reductions and use them to find the Gr\"obner basis of every term $X_kLUQ_iP_jVRY_l$.
Let us fix the $t$'s at the end of the $X_k$ words: $X_k=X'_kt^{q_k}$. These are lexicographical equalities and $q_k\ge 0$.

Since $\sum_{k,l} c_{k,l} X'_kt^{q_k}LUQ_iP_jVRY'_l= 0$, this sum (in the reduced form) can be separated into several sets
of similar monomials. Consider one of these sets:
$\sum X'_ut^{x_u}LUQ_iP_jVRY'_u$. If these monomials are similar then all $X'_u$ must be also similar.
Recall that $\tilde h(X_k)$ does not depend on $k$, therefore  all $x_u$ must be the same.

Hence, $n=1$, and $m>1$ and we have a situation $XLUQ_iP_jVR(\sum_{l=1}^m d_lY_l)=0$, where $X\in\Psi^*$ is a word, and $m$ is minimal.
Therefore, all words $XLUQ_iP_jVRY_l$ should be equal in the algebra, however $Y_l$ should be pairwise different.
If we will reduce word $WRY_l$ (for $W=XLUQ_iP_jV$), only letter $s$ can pass through $R$, therefore the only case 
to reduce it is to pass letters $s$ from $W$ to $Y_l$. The number of these letters $s$ depend on $W$ therefore it will be similar.
Therefore we will have an equality $s^kY_1=s^kY_2=\dots=s^kY_m$ (for some non negative number $k$) in the algebra.
Note that relations with letter $s$ do not change a structure of the word, and one can see that for any two different words 
$Y$ and $Z$, 
words $sY$ and $sZ$ must be also different. Therefore $s^kY_1\neq s^k Y_2$ in the algebra, and $m$ cannot be larger than $1$.

This contradiction completes the proof.

\end{proof}

\medskip

Since the halting problem cannot be algorithmically solved,
the zero divisors problem in algebra $H$ is algorithmically unsolvable.

\medskip

\begin{rem*}
 We can consider two semigroups corresponding to our algebras: 
 in both algebras each relation is written as an equality of two monomials.
 Therefore the same alphabets together with the same sets of relations define semigroups.
 In both semigroups the equality problem is algorithmically solvable, since it is solvable in algebras.
 However in the first semigroup a nilpotency problem is algorithmically unsolvable, 
 and in the second semigroup a zero divisor problem is algorithmically unsolvable.
\end{rem*}

\pagestyle{headings}

\end{document}